\documentclass[10pt]{article}
\usepackage{amssymb}
\usepackage{amsthm}
\usepackage{amsmath}

 \newtheorem{theorem}{Theorem}[section]
 \newtheorem{corollary}[theorem]{Corollary}
 \newtheorem{lemma}[theorem]{Lemma}
 \newtheorem{proposition}[theorem]{Proposition}

 \theoremstyle{definition}
 \newtheorem{definition}[theorem]{Definition}

 \theoremstyle{remark}
 \newtheorem{remark}[theorem]{Remark}

\newcommand{\B}{\mathbb{B}}
\newcommand{\C}{\mathbb{C}}
\newcommand{\D}{\mathbb{D}}

\begin{document}


\title{Bloch-type spaces and extended Ces\`{a}ro operators in the unit ball of a complex Banach space
}

\author{HAMADA Hidetaka 
\\
\\
{\it Faculty of Science and Engineering,
Kyushu Sangyo University,}\\
{\it 3-1 Matsukadai 2-Chome, Higashi-ku,
Fukuoka 813-8503,
Japan}
\\
{\it Email: {h.hamada@ip.kyusan-u.ac.jp}
}
}

%
%



\date{}

\maketitle

\begin{abstract}
Let $\mathbb{B}$ be the unit ball of a complex Banach space $X$.
In this paper, we will generalize the Bloch-type spaces and the little Bloch-type spaces to 
the open unit ball $\mathbb{B}$ by using the radial derivative.
Next, we define an extended Ces\`{a}ro operator $T_{\varphi}$ with holomorphic symbol $\varphi$ and characterize those $\varphi$ for which $T_{\varphi}$ is bounded between the Bloch-type spaces and the little Bloch-type spaces.
We also characterize those $\varphi$ for which $T_{\varphi}$ is compact between the Bloch-type spaces and the little Bloch-type spaces under some additional assumption on the symbol $\varphi$.
When $\mathbb{B}$ is the open unit ball of a finite dimensional complex Banach space $X$, 
this additional assumption is automatically satisfied.
\end{abstract} 

\medskip


\section{Introduction}
\label{intro}
\setcounter{equation}{0}

Let $\D$ denote the unit disc in $\C$.
For a holomorphic function $f(z)=\sum_{k=0}^{\infty}a_kz^k$ on $\D$,
the Ces\`{a}ro operator is defined by
\[
C(f)(z)=\sum_{j=0}^{\infty}\left(\frac{1}{j+1}\sum_{k=0}^{j}a_k\right)z^j.
\]
The boundedness of this operator on some spaces of holomorphic functions
was considered by many authors
(see \cite{M92},  \cite{SR98}, \cite{S87}, \cite{S96}, \cite{X97}).
The integral form of the Ces\`{a}ro operator is
\[
C(f)(z)=\frac{1}{z}\int_0^{z}f(\zeta)\frac{1}{1-\zeta}d\zeta
=\int_0^{1}f(tz)\left(\log\frac{1}{1-\zeta}\right)'|_{\zeta=tz} dt.
\]
As a natural extension of the Ces\`{a}ro operator,
the extended Ces\`{a}ro operator $T_{\varphi}$ with holomorphic symbol $\varphi$
is defined by
\[
T_{\varphi}f(z)=\int_0^z f(\zeta)\varphi'(\zeta)d\zeta.
\]
The boundedness and compactness of this operator on the Hardy space, the Bergman space
and the Bloch type spaces 
have been studied in \cite{AC01}, \cite{AS97}, \cite{WH05}.

Let $\B_n$ be the Euclidean unit ball in $\C^n$ and 
$H(\B_n)$ be the family of holomorphic 
functions on $\B_n$.
Given $\varphi \in H(\B_n)$,
the extended Ces\`{a}ro operator $T_{\varphi}$ with holomorphic symbol $\varphi$
is defined by
\[
T_{\varphi}f(z)=\int_0^1 f(tz)\mathcal{R}\varphi(tz)\frac{1}{t}dt,
\]
where
\[
\mathcal{R}\varphi(z)=\sum_{j=1}^n\frac{\partial \varphi}{\partial z_j}(z)z_j
\]
is the radial derivative of $\varphi$.
The boundedness and the compactness of this operator
on $\alpha$-Bloch spaces, little $\alpha$-Bloch spaces
and the Bergman space have been studied in
\cite{H03}, 
\cite{H04},
\cite{S05},
\cite{X04}.
Tang \cite{T07}
characterized those holomorphic symbols $\varphi$
in the Euclidean unit ball of $\mathbb{C}^n$
for which the induced extended Ces\`{a}ro operator 
$T_{\varphi}$ is bounded or compact on the Bloch-type spaces and the little Bloch-type spaces.

On the other hand, 
Wicker \cite{W} and
Blasco, Galindo and Miralles \cite{BGM14}
generalized the Bloch space to the unit ball of an infinite dimensional complex Hilbert space.
Deng and Ouyang \cite{DO06} and
Chu, Hamada, Honda and Kohr \cite{CHHKb} independently generalized the Bloch space to 
an infinite dimensional bounded symmetric domain realized as the open unit ball of a JB*-triple $X$
and studied the boundedness and the compactness of composition operators between the Bloch spaces on bounded symmetric domains.
Blasco, Galindo, Lindstr\"om and Miralles \cite{BGLM17}
provided necessary and sufficient conditions for compactness of composition operators
on the space of Bloch functions on the unit ball of a complex Hilbert space
with additional compactness assumptions on the set related to the composition symbol.
Further, Hamada \cite{H17} studied the weighted composition operators from the Hardy space $H^{\infty}$ to the Bloch space on bounded symmetric domains.

In this paper, 
we will generalize the Bloch-type spaces and the little Bloch-type spaces to 
the open unit ball $\mathbb{B}$ of a general infinite dimensional complex Banach space $X$
by using the radial derivative.
Our definition is new, but if $X$ is a complex Hilbert space, it is equivalent to the definition which is an extension of
that in the finite dimensional case.
Next, we define an extended Ces\`{a}ro operator $T_{\varphi}$ with holomorphic symbol $\varphi$ and characterize those $\varphi$ for which $T_{\varphi}$ is bounded between the Bloch-type spaces and the little Bloch-type spaces.
As in  \cite{BGLM17}, under some additional assumption on the symbol $\varphi$,
we also characterize those $\varphi$ for which $T_{\varphi}$ is compact between the Bloch-type spaces and the little Bloch-type spaces.
When $\mathbb{B}$ is the open unit ball of a finite dimensional complex Banach space $X$, 
this additional assumption is automatically satisfied.
There are some gaps in \cite{T07}.
We overcome these gaps and give a complete proof in this paper
in the setting of the unit ball of a general infinite dimensional complex Banach space.

\section{Bloch-type spaces and little Bloch-type spaces}
\label{section-Bloch-type}
\setcounter{equation}{0}

A positive continuous function $\omega$ on $[0,1)$ is said to be normal
if there are constants $\delta\in [0,1)$ and $0<a<b<\infty$
such that
\begin{equation}
\label{omega}
\frac{\omega(r)}{(1-r)^a}
\mbox{ is decreasing}
\quad
\mbox{and}
\quad
\frac{\omega(r)}{(1-r)^b}
\mbox{ is increasing}
\quad
\mbox{on }
[\delta, 1).
\end{equation}

Then a normal function $\omega$ is strictly decreasing on $[\delta,1)$ and 
$\omega(r) \to 0$ as $r\to 1$.

Let $\D$ be the unit disc in $\C$.

\begin{lemma}
\label{lem-g}
Let $\omega$ be a normal function.
Denote 
$k_0=\max(0,[\log_2\frac{1}{\omega(\delta)}])$,
$r_k=(\omega|_{[\delta, 1)})^{-1}(\frac{1}{2^k})$
and
$n_k=[\frac{1}{1-r_k}]$
for $k>k_0$,
where the symbol $[x]$ means the greatest integer not more than $x$.
Let
\[
g(\zeta)=1+\sum_{k>k_0}^{\infty}2^k\zeta^{n_k},
\quad
\zeta \in \mathbb{D}.
\]
Then
\begin{enumerate}
\item[$({\rm i})$]
$g$ is a holomorphic function on $\mathbb{D}$
such that
$g(r)$ is increasing on $[0,1)$ and 
\[
0<C_1=\inf_{r\in [0,1)}\omega(r)g(r)\leq
\sup_{r\in [0,1)}\omega(r)g(r)=C_2
<\infty;
\]
\item[$({\rm ii})$]
there exists a positive constant $C_3$ such that the inequality
\[
\int_0^rg(t)dt\leq C_3\int_0^{r^2}g(t)dt
\]
holds for all $r\in [r_1,1)$,
where $r_1\in (0,1)$ is a constant such that
\[
\int_0^{r_1}
g(t)dt=1.
\]
\end{enumerate}
\end{lemma}

\begin{proof}
(i) was proved in \cite[Theorem 2.3]{HW05}.
We give a proof for (ii).
Let $\delta$ be the constant in \eqref{omega}.
We may assume that $r_1<\delta^{1/4}$.
First, we consider the case $r\in [r_1,\delta^{1/4}]$.
Then $\int_0^rg(t)dt$ is bounded above and 
$\int_0^{r^2}g(t)dt$ is bounded below by a positive constant.
So, there exists a constant $C>0$ such that
\[
\int_0^rg(t)dt\leq C\int_0^{r^2}g(t)dt,
\quad
r\in[r_1,\delta^{1/4}].
\]
Next, we consider the case $r\in(\delta^{1/4},1)$.
In this case, by (i) and \eqref{omega}, we have
\begin{align*}
\int_{r^2}^{r}g(t)dt
&\leq
C_2\int^r_{r^2}\frac{1}{\omega(t)}dt
\\
&=
C_2\int^r_{r^2}\frac{(1-t)^b}{\omega(t)}\frac{1}{(1-t)^b}dt
\\
&\leq
C_2\frac{(1-r^2)^b}{\omega(r^2)}\frac{r-r^2}{(1-r)^b}
\\
&=
C_2\frac{(1-r^2)^b}{\omega(r^2)}\frac{(r-r^2)(1+r)^b(1+r^2)^b}{(1-r^4)^b}
\\
&\leq
C_2\frac{(r-r^2)(1+r)^b(1+r^2)^b}{r^2-r^4}\int_{r^4}^{r^2}\frac{(1-t)^b}{\omega(t)}\frac{1}{(1-t)^b}dt
\\
&\leq
\frac{C_2(1+r)^b(1+r^2)^b}{C_1(r+r^2)}\int_{r^4}^{r^2}g(t)dt.
\end{align*}
Therefore, there exists a constant $C'>0$ such that
\[
\int_0^rg(t)dt=\int_0^{r^2}g(t)dt+\int_{r^2}^{r}g(t)dt\leq C'\int_0^{r^2}g(t)dt,
\quad
r\in(\delta^{1/4},1).
\]
This completes the proof.
\end{proof}

\begin{remark}
In \cite[eq.(3.5)]{T07},
it is claimed that there exists a constant $C>0$ such that
\begin{equation}
\label{g-mistake}
\int_0^{\| w\|}
g(t)dt\leq
C
\int_0^{\| w\|^2}
g(t)dt,
\quad 
w\in \B.
\end{equation}
However, this is impossible for small $\| w\|$, because $g(0)=1$.
\end{remark}

Let $\B$ be the unit ball of a complex Banach space $X$
with norm $\| \cdot\|$.
A normal function $\omega$ will be extended to a function on 
$\B$ by $\omega(z)=\omega(\| z\|)$.
Let $H(\B)$ denote the set of holomorphic mappings from $\B$ into $\C$.

\begin{definition}
Let $\B$ be the open unit ball of a complex Banach space $X$
and let $\omega$ be a normal function on $\B$.
A function $f\in H(\B)$ is called a Bloch-type function
with respect to $\omega$
if
\[
\| f\|_{\mathcal{B}_{\mathcal{R}}(\B),\omega}
=\sup \{ \omega(z)| \mathcal{R}f(z)| : z\in \B\}<+\infty,
\]
where
$
\mathcal{R}f(z)=
Df(z)z
$
and $Df(z)$ is the Fr\'{e}chet derivative of $f$ at $z$.
\end{definition}

The class of all Bloch-type functions with respect to $\omega$
on $\B$ is
called a Bloch-type space on $\B$ and is denoted by
$\mathcal{B}_{\mathcal{R}}(\B)_{\omega}$. 
Then
\[\| f\|_{{\mathcal{R}},\omega}=|f(0)|+\| f\|_{\mathcal{B}_{\mathcal{R}}(\B),\omega}\]
is a norm on $\mathcal{B}_{\mathcal{R}}(\B)_{\omega}$. 

The following proposition is a generalization of
the result on the Euclidean unit ball in $\C^n$
\cite[Lemma 3.1]{T07}
to the unit ball of a complex Banach space.

\begin{proposition}
\label{p.growth}
Let $\omega$ be a normal function.
Then there exists a constant $C_4>0$ such that
\[
|f(z)|\leq
C_4\left(
1+\int_0^{\| z\|}
\frac{1}{\omega(t)}dt
\right)
\| f\|_{{\mathcal{R}},\omega}
\]
for $f\in \mathcal{B}_{\mathcal{R}}(\B)_{\omega}$
and $z\in \B$.
\end{proposition}

\begin{proof}
First we consider the case $\| z\|<1/2$.
Since $\mathcal{R}f(0)=0$
and
\[
|\mathcal{R}f(z)|\leq 
\frac{\| f\|_{\mathcal{B}_{\mathcal{R}}(\B),\omega}}
{\min_{t\in[0,1/2]}\omega(t)},
\quad
\| z\|<\frac{1}{2},
\]
we have
\[
|\mathcal{R}f(z)|\leq 
\frac{2\| f\|_{\mathcal{B}_{\mathcal{R}}(\B),\omega}}
{\min_{t\in[0,1/2]}\omega(t)}\| z\|,
\quad
\| z\|<\frac{1}{2}
\]
by the Schwarz lemma.
Note that ${\min_{t\in[0,1/2]}\omega(t)}>0$, 
since $\omega$ is a positive continuous function on $[0,1)$.
Therefore, we have
\begin{align}
|f(z)|
&\leq 
|f(0)|+|f(z)-f(0)|
\nonumber \\
&\leq
|f(0)|+\int_{0}^{1}
\left| \frac{\mathcal{R}f(tz)}{t}
\right|dt
\nonumber \\
&\leq
|f(0)|+
\frac{2\| z\|}
{\min_{t\in[0,1/2]}\omega(t)}
\| f\|_{\mathcal{B}_{\mathcal{R}}(\B),\omega}.
\label{eq2.5}
\end{align}

Next, let $z\in \B$ with $\| z\|\geq 1/2$.
Then, applying \eqref{eq2.5} at the point $\frac{z}{2}$, we have
\begin{align*}
|f(z)|
&\leq
\left| f\left(
\frac{z}{2}
\right)\right|
+
\left|
f(z)-f\left(\frac{z}{2}\right)
\right|
\\
&\leq
\left| f\left(
\frac{z}{2}
\right)\right|
+
\int_{1/2}^{1}
\left| \frac{\mathcal{R}f(tz)}{t}
\right|dt
\\
&\leq
\left| f\left(
\frac{z}{2}
\right)\right|
+
4\| f\|_{\mathcal{B}_{\mathcal{R}}(\B),\omega}
\int_{1/2}^1\frac{\| z\|}{\omega(t\|z\|)}dt
\\
&\leq
|f(0)|+
\frac{\| z\|}
{\min_{t\in[0,1/2]}\omega(t)}
\| f\|_{\mathcal{B}_{\mathcal{R}}(\B),\omega}
+
4\| f\|_{\mathcal{B}_{\mathcal{R}}(\B),\omega}
\int_0^{\| z\|}
\frac{1}{\omega(t)}dt.
\end{align*}

\end{proof}

\begin{proposition}
\label{p.banach}
The Bloch-type space 
$\mathcal{B}_{\mathcal{R}}(\B)_{\omega}$ 
is a complex Banach space with
 the norm $\| f\|_{{\mathcal{R}},\omega}$.
\end{proposition}

\begin{proof}
Let $(f^k)$ be a Cauchy sequence in 
$\mathcal{B}_{\mathcal{R}}(\B)_{\omega}$. 
By
Proposition \ref{p.growth}, it follows that $(f^k)$ is a Cauchy sequence in the space $H(\B)$,
where $H(\B)$ is equipped with the locally uniform topology.
Hence $(f^k)$ converges locally uniformly to some function
$f\in H(\B)$.

To complete the proof, we show $\|f^k-f\|_{{\mathcal{R}},\omega}\to 0$ as $k\to\infty$. For this, fix $\varepsilon>0$.
Since $(f^k)$ is a Cauchy sequence in $\mathcal{B}_{\mathcal{R}}(\B)_{\omega}$, there exists $k_0\in\mathbb{N}$
such that
$$\|f^k-f^p\|_{{\mathcal{R}},\omega}<\varepsilon \quad {\rm for} \quad k,p\geq k_0$$
which gives
$$|f^k(0)-f^p(0)|+\omega(z)|(Df^k(z)-Df^p(z))z|<\varepsilon \quad (z\in \B,\, k,p\geq k_0).$$

On the other hand, given $p\in\mathbb{N}$ and $z\in \B$, 
the locally uniform convergence of the sequence
 $(f^k)$ to $f$ implies that
$$|f(0)-f^p(0)|+\omega(z)|(Df(z)-Df^p(z))z|
\leq\varepsilon$$
for $p\geq k_0$ and $z\in \B$. Consequently,
$$\|f^p-f\|_{{\mathcal{R}},\omega}\leq \varepsilon \quad {\rm for} \quad p\geq k_0.$$
Therefore $f=(f-f^p)+f^p\in  \mathcal{B}_{\mathcal{R}}(\B)_{\omega}$
and $\lim_{p\to\infty}\|f^p-f\|_{{\mathcal{R}},\omega}=0$.
This proves that
$\mathcal{B}_{\mathcal{R}}(\B)_{\omega}$ is complete.
\end{proof}

A function $f\in H(\B)$
is said to belong to  the little Bloch-type space $\mathcal{B}_{\mathcal{R}}(\B)_{\omega,0}$
if 
\[
\lim_{\| z\|\to 1}\omega(z)|\mathcal{R}f(z)|=0
\]
holds.
Since for each $R\in (0,1)$, there exists a constant $C(R)>1$ such that
\begin{equation}
\label{eq2.10}
\sup_{0\leq r\leq R}\omega(r)\leq C(R)\omega(R),
\end{equation}
$\mathcal{B}_{\mathcal{R}}(\B)_{\omega,0}$
is a closed subspace of $\mathcal{B}_{\mathcal{R}}(\B)_{\omega}$.

For $x\in X\setminus\{0\}$, we define
$$T(x)=\{l_x\in X^*:\ l_x(x)=\|x\|,\ \|l_x\|=1\}.$$
Then $T(x)\ne \emptyset$ in view of the Hahn-Banach theorem.

Now, we generalize the test functions defined in \cite{T07} on the Euclidean unit ball of $\C^n$
to the unit ball of a complex Banach space.
These test functions will be useful in the next sections.

\begin{lemma}
\label{testfunction1}
Let $g\in H(\mathbb{D})$ be the function defined in Lemma $\ref{lem-g}$.
For each $v\in \B\setminus\{ 0\}$
and $l_v\in T(v)$, let
\[
f_v(z)=\int_0^{\| v\|l_v(z)}g(\zeta)d\zeta,
\quad
z\in \B.
\]
Then $f_v\in \mathcal{B}_{\mathcal{R}}(\B)_{\omega,0}$
and
$\| f_v\|_{{\mathcal{R}},\omega}\leq C_2$,
where $C_2$ is the constant defined in Lemma $\ref{lem-g}$.
\end{lemma}

\begin{proof}
By Lemma \ref{lem-g} (i), we have 
\[
\omega(z) |\mathcal{R}f_v(z)|=\omega(z)|g(\| v\|l_v(z))| \| v\| |l_v(z)|
\leq \omega(\| z\|)g(\| z\|)\leq C_2.
\]
Therefore, 
$f_v\in \mathcal{B}_{\mathcal{R}}(\B)_{\omega}$
and $\| f_v\|_{{\mathcal{R}},\omega}\leq C_2$.
Moreover, 
since $\mathcal{R}f_v$ is bounded on $\B$
and $\omega(z)\to 0$ as $\| z\|\to 1$,
we have
$f_v\in \mathcal{B}_{\mathcal{R}}(\B)_{\omega,0}$.
\end{proof}

\begin{lemma}
\label{testfunction2}
For each $v\in \B$ with $\| v\|\geq r_1$, let
\[
F_v(z)=\frac{1}{f_v(v)}\left( f_v(z) \right)^2,
\quad
z\in \B,
\]
where $r_1$ is the constant in Lemma $\ref{lem-g}$
and $f_v$ is the function defined in Lemma $\ref{testfunction1}$.
Then $F_v\in \mathcal{B}_{\mathcal{R}}(\B)_{\omega,0}$
and
$\| F_v\|_{{\mathcal{R}},\omega}\leq 2C_2C_3$,
where $C_2$, $C_3$ are the constants defined in Lemma $\ref{lem-g}$.
Moreover, if $\int_0^1\frac{1}{\omega(t)}dt=\infty$,
then $F_v\to 0$ uniformly on any closed ball strictly inside $\B$ as $\| v\|\to 1$.
\end{lemma}

\begin{proof}
By Lemma \ref{lem-g}, we have 
\begin{align*}
\omega(z) |\mathcal{R}F_v(z)|
&=
\omega(z)\frac{2}{f_v(v)}|f_v(z)g(\| v\|l_v(z))| \| v\| |l_v(z)|
\\
&\leq
2 \frac{\int_0^{\| v\|}g(t)dt}{\int_0^{\| v\|^2}g(t)dt}\omega(\| z\|)g(\| z\|)
\\
&\leq 
2C_2C_3.
\end{align*}
Therefore, 
$F_v\in \mathcal{B}_{\mathcal{R}}(\B)_{\omega}$
and $\| F_v\|_{{\mathcal{R}},\omega}\leq 2C_2C_3$.
Moreover, 
since $\mathcal{R}F_v$ is bounded on $\B$
and $\omega(z)\to 0$ as $\| z\|\to 1$,
we have
$F_v\in \mathcal{B}_{\mathcal{R}}(\B)_{\omega,0}$.

Next, assume that $\int_0^1\frac{1}{\omega(t)}dt=\infty$.
Fix $r\in (0,1)$.
Since
\[
f_v(v)=\int_0^{\| v\|^2}g(t)dt\geq \int_0^{\| v\|^2}\frac{C_1}{\omega(t)}dt
\to \infty 
\quad
\mbox{as }
\| v\|\to 1
\]
and
\[
|f_v(z)|\leq \int_0^{r}g(t)dt,
\quad
\| z\|\leq r,
\]
$F_v(z)\to 0$ 
uniformly for $\| z\|\leq r$ as $\| v\|\to 1$.
\end{proof}

Let $f\in H(\B)$.
Then the relation $|\mathcal{R}f(z)|\leq \| Df(z)\|$ holds.
So, 
if
\[
\sup_{z\in\B}\omega(z)\| Df(z)\|<\infty
\]
holds,
then $f\in \mathcal{B}_{\mathcal{R}}(\B)_{\omega}$
and
\[
\| f\|_{{\mathcal{R}},\omega}\leq 
|f(0)|+\sup_{z\in\B}\omega(z)\| Df(z)\|
\]
holds.
In the case $\B=\B_H$ is the unit ball of a complex Hilbert space $H$, we have the following theorem,
which is a generalization of
the result on the Euclidean unit ball in $\C^n$
\cite[Theorem 2.1]{T07}
to the unit ball of a complex Hilbert space.
Note that there is a gap in the proof of \cite[Theorem 2.1]{T07},
because \cite[eq.(2.7)]{T07} cannot be obtained from
\cite[eq.(2.2)]{T07}.
To overcome this gap,
we will change the path of integration of Cauchy's integral formula.

\begin{theorem}
\label{Hilbert}
Let $\B_H$ be the unit ball of a complex Hilbert space $H$ and 
let $\omega$ be a normal function.
Let $f\in H(\B_H)$.
Then
\begin{enumerate}
\item[$({\rm i})$]
$f\in \mathcal{B}_{\mathcal{R}}(\B_H)_{\omega}$
if and only if
$\sup_{z\in\B_H}\omega(z)\| Df(z)\|<\infty$.
Moreover, 
\[
\| f\|_{{\mathcal{R}},\omega}
\simeq
|f(0)|+\sup_{z\in\B_H}\omega(z)\| Df(z)\|;
\]
\item[$({\rm ii})$]
$f\in \mathcal{B}_{\mathcal{R}}(\B_H)_{\omega,0}$
if and only if
$\lim_{\| z\| \to 1}\omega(z)\| Df(z)\|=0.$
\end{enumerate}
\end{theorem}

\begin{proof}
(i)
We may assume that $\dim H\geq 2$.
It suffices to show that
there exists a constant $C>0$ such that
\begin{equation}
\label{eq(I)}
\sup_{z\in\B_H}\omega(z)| Df(z)v|\leq
C\sup_{z\in\B_H}\omega(z)|\mathcal{R}f(z)|,
\quad f\in \mathcal{B}_{\mathcal{R}}(\B_H)_{\omega},
\| v\|=1.
\end{equation}
Let $z\in \B_H$ and $v \in H$ with $\| v\|=1$ be fixed.
Then there exist orthonormal unit vectors $e_1$, $e_2\in H$ and $\alpha,\beta_1,\beta_2\in \C$
with $|\alpha|<1$ and $|\beta_1|^2+|\beta_2|^2=1$
such that $z=\alpha e_1$, $v=\beta_1 e_1+\beta_2 e_2$.
For $f\in \mathcal{B}_{\mathcal{R}}(\B_H)_{\omega}$, let
\[
F(z_1,z_2)=f(z_1 e_1+z_2 e_2),
\quad
(z_1, z_2)\in \B_2,
\]
where $\B_2$ is the Euclidean unit ball in $\C^2$.
Then $F\in H(\B_2)$ and 
$\mathcal{R}F(z_1,z_2) =\mathcal{R}f(z_1 e_1+z_2 e_2)$,
$\frac{\partial F}{\partial z_1}(z_1,0)=Df(z_1 e_1)e_1$,
$\frac{\partial F}{\partial z_2}(z_1,0)=Df(z_1 e_1)e_2$
hold.
Let $R\in (\delta,1)$ be fixed.
We assume that $|z_1|\geq R$ and let $r=|z_1|$.
Since
$\delta<R\leq \sqrt{t^2+R^2(1-r^{-2}t^2)}\leq r$ for $0\leq t\leq r$
and $\omega$ is strictly decreasing on $[\delta,1)$,
we have
\[
\omega\left(\sqrt{t^2+R^2(1-r^{-2}t^2)}\right)\geq \omega(r),
\quad 0\leq t\leq r.
\]
Then, for $0\leq t< r$,
by Cauchy's integral formula,
we have
\begin{align}
\left|\frac{\partial (\mathcal{R}F)}{\partial z_2}(t,0)\right|
&=
\left|\frac{1}{2\pi i}\int_{|z_2|=R\sqrt{1-r^{^-2}t^2}}\frac{\mathcal{R}F(t,z_2)}{z_2^2}dz_2 \right|
\nonumber \\
&=
\left| \frac{1}{2\pi}\int_{|z_2|=
R\sqrt{1-r^{-2}t^2}}\frac{\mathcal{R}f(t e_1+z_2 e_2)}{z_2^2}dz_2\right|
\nonumber \\
&\leq
\frac{\max_{|z_2|=R\sqrt{1-r^{-2}t^2}}|\mathcal{R}f(t e_1+z_2 e_2)|}{R\sqrt{1-r^{-2}t^2}}
\nonumber \\
&\leq
\frac{\sup_{R\leq \| z\|<1}\omega(z)|\mathcal{R}f(z)|}{\omega(r)R\sqrt{1-r^{-2}t^2}}.
\label{eq2.6}
\end{align}
Therefore, for $|z_1|=r\geq R$, 
by \cite[Lemma 6.4.5(2)]{R} and \eqref{eq2.6},
we have
\begin{align*}
|z_1|\left|\frac{\partial F}{\partial z_2}(z_1,0)\right|
&=
\left| \int_0^{r}\frac{\partial (\mathcal{R}F)}{\partial z_2}(t,0)dt \right|
\\
&\leq
\int_0^{r} \frac{\sup_{R\leq \| z\|<1}\omega(z)|\mathcal{R}f(z)|}{\omega(r)R\sqrt{1-r^{-2}t^2}} dt
\\
&=
\frac{\sup_{R\leq \| z\|<1}\omega(z)|\mathcal{R}f(z)|}{\omega(r)R}
\int_0^{r} \frac{1}{\sqrt{1-r^{-2}t^2}} dt
\\
&=
\frac{\pi}{2\omega(|z_1|)R}|z_1|\sup_{R\leq \| z\|<1}\omega(z)|\mathcal{R}f(z)|.
\end{align*}
Thus, we have
\begin{equation}
\label{eq2.7}
\left|\frac{\partial F}{\partial z_2}(z_1,0)\right|
\leq
\frac{\pi}{2\omega(|z_1|)\delta}\sup_{R\leq \| z\|<1}\omega(z)|\mathcal{R}f(z)|,
\quad
|z_1|\geq  R.
\end{equation}
Also, we have
\begin{equation}
\label{eq2.8}
\left|\frac{\partial F}{\partial z_1}(z_1,0)\right|
=
\left|\frac{\mathcal{R}f(z_1 e_1)}{z_1}\right|
\leq
\frac{1}{\delta \omega(|z_1|)}\sup_{R\leq \| z\|<1}\omega(z)|\mathcal{R}f(z)|,
\quad
|z_1|\geq R.
\end{equation}
From \eqref{eq2.7} and \eqref{eq2.8}, we have
\begin{align}
\omega(z)|Df(z)v|
&=
\omega(\alpha)\left|Df(\alpha e_1)({\beta_1} e_1+\beta_2 e_2)\right|
\nonumber \\
&=
\omega(\alpha)\left|
\beta_1\frac{\partial F}{\partial z_1}(\alpha,0)+\beta_2 \frac{\partial F}{\partial z_2}(\alpha,0)
\right|
\nonumber \\
&\leq 
\omega(\alpha)\left(
\left|\frac{\partial F}{\partial z_1}(\alpha,0)\right|^2+
\left|\frac{\partial F}{\partial z_2}(\alpha,0)\right|^2
\right)^{1/2}
\nonumber \\
&\leq 
\frac{\pi}{\sqrt{2}\delta}\sup_{R\leq \| z\|<1}\omega(z)|\mathcal{R}f(z)|,
\quad 
\| z\|\geq  R, \| v\|=1.
\label{little}
\end{align}
Since $Df(z)v$ is a holomorphic function in $z\in \B$,
by \eqref{eq2.10}, \eqref{little} and the maximum principle
for holomorphic functions,
we have
\[
\omega(z)|Df(z)v|\leq\frac{\pi}{\sqrt{2}\delta}C(R)\sup_{z\in \B}\omega(z)|\mathcal{R}f(z)|,
\quad 
z\in\B, \| v\|=1.
\]
This implies \eqref{eq(I)}.

(ii)
It suffices to show that $f\in \mathcal{B}_{\mathcal{R}}(\B_H)_{\omega,0}$ implies
\begin{equation}
\label{eq(II)}
\lim_{\| z\| \to 1}\omega(z)\| Df(z)\|=0.
\end{equation}
Assume that the condition
$\lim_{\| z\|\to 1}\omega(z)|\mathcal{R}f(z)|=0$
holds.
Then 
for any $\varepsilon>0$,
there exists $R\in (\delta,1)$ such that
\[
\omega(z)|\mathcal{R}f(z)|<\varepsilon,
\quad \| z\|>R_.
\]
Therefore, by using \eqref{little}, we obtain \eqref{eq(II)}.
This completes the proof.
\end{proof}

\section{Boundedness of extended Ces\`{a}ro operators}
\label{section-bounded}
\setcounter{equation}{0}
Given $\varphi \in H(\B)$,
the extended Ces\`{a}ro operator $T_{\varphi}$ is defined by
\[
T_{\varphi}f(z)=\int_0^1f(tz)\mathcal{R}\varphi(tz)\frac{1}{t}dt,
\quad
f\in H(\B), z\in \B.
\] 

The following lemma is a generalization of
the result on the Euclidean unit ball in $\C^n$
\cite[Lemma 2.1]{S05}
to the unit ball of a complex Banach space.

\begin{lemma}
\label{Taylor}
For every $f, \varphi \in H(\B)$, it holds that
\[
\mathcal{R}[T_{\varphi}f](z)=f(z)\mathcal{R}\varphi(z).
\]
\end{lemma}

\begin{proof}
$f\mathcal{R}\varphi \in H(\B)$ has
the Taylor series $f(z)\mathcal{R}\varphi(z)=\sum_{n=1}^{\infty} P_n(z)$, where $P_n$ is a homogeneous polynomial
 of degree $n$.
Then we have
\begin{align*}
\mathcal{R}[T_{\varphi}f](z)
&=
\mathcal{R}\int_0^1\sum_{n=1}^{\infty} P_n(z)t^n\frac{1}{t}dt
\\
&=
\mathcal{R}\sum_{n=1}^{\infty} \frac{P_n(z)}{n}
\\
&=
\sum_{n=1}^{\infty} P_n(z)
\\
&=
f(z)\mathcal{R}\varphi(z).
\end{align*}
\end{proof}

Tang \cite[Theorems 3.1 and 3.2]{T07} obtained the following theorems when 
$\B$ is the Euclidean unit ball of $\mathbb{C}^n$.
The following theorems are generalization
to the unit ball of a complex Banach space.
Note that the proof in \cite[Theorem 3.1]{T07}
has a gap, because 
\eqref{g-mistake}
is used in it.
For non-negative constants $A_{\lambda}$ and $B_{\lambda}$
with a parameter $\lambda$,
the expression $A_{\lambda}\simeq B_{\lambda}$ means
that there exists a constant $C>0$ which is independent of $\lambda$
such that
$C^{-1}A_{\lambda}\leq B_{\lambda} \leq CA_{\lambda}$.

\begin{theorem}
\label{bounded}
Let $\omega$ and $\mu$ be normal functions.
Let $\varphi\in H(\B)$.
Then $T_{\varphi}: \mathcal{B}_{\mathcal{R}}(\B)_{\omega}
\to \mathcal{B}_{\mathcal{R}}(\B)_{\mu}$
is bounded if and only if 
\begin{equation}
\label{eq-bounded}
\sup_{z\in \B}\mu(z)|\mathcal{R}\varphi(z)|\int_0^{\| z\|}
\frac{1}{\omega(t)}dt<\infty.
\end{equation}
Moreover, if $T_{\varphi}: \mathcal{B}_{\mathcal{R}}(\B)_{\omega}
\to \mathcal{B}_{\mathcal{R}}(\B)_{\mu}$
is bounded,
then
\begin{equation}
\label{eq-bound2}
\| T_{\varphi}\|\simeq\sup_{z\in \B}\mu(z)|\mathcal{R}\varphi(z)|\int_0^{\| z\|}
\frac{1}{\omega(t)}dt.
\end{equation}
\end{theorem}

\begin{proof}
Assume that \eqref{eq-bounded} holds.
Let $\eta\in (0,1)$ be such that
$|\mathcal{R}\varphi(z)|\leq 1$  for $\| z\|\leq \eta$.
There exists $C_5>0$ such that
\begin{equation}
\label{eq3.3}
1\leq C_5\int_0^{\eta}\frac{1}{\omega(t)}dt.
\end{equation}
Then, there exists $C_6>0$ such that
\begin{equation}
\label{eq3.4}
\sup_{\| z\|\leq \eta}\mu(z)|\mathcal{R}\varphi(z)|
\left(
1+\int_0^{\| z\|}
\frac{1}{\omega(t)}dt
\right)
\leq
C_6
\sup_{\| z\|\geq \eta}\mu(z)|\mathcal{R}\varphi(z)|
\int_0^{\| z\|}
\frac{1}{\omega(t)}dt.
\end{equation}
Let $C_7=\max\{ C_5+1, C_6\}$.
Then, by Proposition \ref{p.growth}, Lemma \ref{Taylor},
\eqref{eq3.3} and \eqref{eq3.4},
we have
\begin{align*}
\mu(z)|\mathcal{R}(T_{\varphi}f)(z)|
&=
\mu(z)|f(z)||\mathcal{R}\varphi(z)|
\\
&\leq
C_4\mu(z)|\mathcal{R}\varphi(z)|
\left(
1+\int_0^{\| z\|}
\frac{1}{\omega(t)}dt
\right)
\| f\|_{{\mathcal{R}},\omega}
\\
&\leq
C_4C_7\| f\|_{{\mathcal{R}},\omega}
\sup_{\| z\|\geq \eta}\mu(z)|\mathcal{R}\varphi(z)|
\int_0^{\| z\|}
\frac{1}{\omega(t)}dt
\\
&\leq 
C_4C_7\| f\|_{{\mathcal{R}},\omega}
\sup_{z\in\B}\mu(z)|\mathcal{R}\varphi(z)|
\int_0^{\| z\|}
\frac{1}{\omega(t)}dt
\end{align*}
for $f\in \mathcal{B}_{\mathcal{R}}(\B)_{\omega}$ and $z\in \B$.
Since $(T_{\varphi}f)(0)=0$, we obtain
that $T_{\varphi}: \mathcal{B}_{\mathcal{R}}(\B)_{\omega}
\to \mathcal{B}_{\mathcal{R}}(\B)_{\mu}$
is bounded and
\begin{equation}
\label{bound1}
\| T_{\varphi}\|
\leq
C_4C_7
\sup_{z\in\B}\mu(z)|\mathcal{R}\varphi(z)|
\int_0^{\| z\|}
\frac{1}{\omega(t)}dt.
\end{equation}

Conversely, assume that 
$T_{\varphi}: \mathcal{B}_{\mathcal{R}}(\B)_{\omega}
\to \mathcal{B}_{\mathcal{R}}(\B)_{\mu}$
is bounded.
Then $\varphi(z)=\varphi(0)+\int_0^1 \mathcal{R}\varphi(tz)\frac{1}{t}dt
=\varphi(0)+(T_{\varphi}1)(z)\in  \mathcal{B}_{\mathcal{R}}(\B)_{\mu}$.
Let $v\in \B\setminus \{ 0\}$ be fixed
and let $f_v\in \mathcal{B}_{\mathcal{R}}(\B)_{\omega,0}$
be the function defined in Lemma \ref{testfunction1}.
Let $r_1$ be the constant in Lemma \ref{lem-g}.
If $\| v\|\geq r_1$, by Lemmas \ref{lem-g} and \ref{testfunction1},
we have
\begin{align}
\mu(v)|\mathcal{R}\varphi(v)|\int_0^{\| v\|}\frac{1}{\omega(t)}dt
&\leq
\mu(v)|\mathcal{R}\varphi(v)|\int_0^{\| v\|}\frac{g(t)}{C_1}dt
\nonumber \\
&\leq 
\frac{C_3}{C_1}\mu(v)|\mathcal{R}\varphi(v)|\int_0^{\| v\|^2}{g(t)}dt
\nonumber \\
&\leq 
\frac{C_3}{C_1}\| T_{\varphi}f_v\|_{\mathcal{R}, \mu}
\nonumber \\
&\leq 
\frac{C_2C_3}{C_1}\| T_{\varphi}\|<\infty.
\label{bound2}
\end{align}
If $\| v\|<r_1$, then by Lemma \ref{lem-g}, we have
\begin{align}
\mu(v)|\mathcal{R}\varphi(v)|\int_0^{\| v\|}\frac{1}{\omega(t)}dt
&\leq
\mu(v)|\mathcal{R}\varphi(v)|\int_0^{\| v\|}\frac{g(t)}{C_1}dt
\nonumber \\
&\leq 
\frac{1}{C_1}\mu(v)|\mathcal{R}\varphi(v)|
\nonumber \\
&\leq 
\frac{1}{C_1}\| T_{\varphi}1\|_{\mathcal{R}, \mu}
\nonumber \\
&\leq 
\frac{1}{C_1}\| T_{\varphi}\|<\infty.
\label{bound30}
\end{align}
The inequalities \eqref{bound2} and \eqref{bound30} yield \eqref{eq-bounded},
as desired.

Moreover, from \eqref{bound1}, \eqref{bound2} and  \eqref{bound30},
we obtain \eqref{eq-bound2}.
This completes the proof.
\end{proof}

\begin{theorem}
\label{bounded-little}
Let $\omega$ and $\mu$ be normal functions.
Let $\varphi \in H(\B)$.
Then $T_{\varphi}: \mathcal{B}_{\mathcal{R}}(\B)_{\omega,0}
\to \mathcal{B}_{\mathcal{R}}(\B)_{\mu,0}$
is bounded if and only if 
$\varphi \in \mathcal{B}_{\mathcal{R}}(\B)_{\mu,0}$ and
\begin{equation}
\label{boundedcondition}
\sup_{z\in \B}\mu(z)|\mathcal{R}\varphi(z)|\int_0^{\| z\|}
\frac{1}{\omega(t)}dt<\infty.
\end{equation}
\end{theorem}

\begin{proof}
Assume that
$\varphi \in \mathcal{B}_{\mathcal{R}}(\B)_{\mu,0}$ and
\[
M=\sup_{z\in \B}\mu(z)|\mathcal{R}\varphi(z)|\int_0^{\| z\|}
\frac{1}{\omega(t)}dt<\infty.
\]
Then $T_{\varphi}: \mathcal{B}_{\mathcal{R}}(\B)_{\omega}
\to \mathcal{B}_{\mathcal{R}}(\B)_{\mu}$
is bounded by Theorem \ref{bounded}.
Therefore,
it suffices to show that
$T_{\varphi}(f)\in \mathcal{B}_{\mathcal{R}}(\B)_{\mu,0}$
for any $f\in \mathcal{B}_{\mathcal{R}}(\B)_{\omega,0}$.
To this end, let $f\in \mathcal{B}_{\mathcal{R}}(\B)_{\omega,0}$
be arbitrarily fixed.
Let $\varepsilon>0$ be fixed.
Then there exists $r_0\in (1/2,1)$ such that
\begin{equation}
\label{eq3.9}
\omega(z)|\mathcal{R}f(z)|<\frac{\varepsilon}{4M},
\quad r_0\leq \| z\|<1.
\end{equation}
For any $z\in \B$ with $r_0<\| z\|<1$,
let $\hat{z}=r_0z/\| z\|$.
Then, by \eqref{eq3.9}, we have
\begin{align*}
|f(z)-f(\hat{z})|
&=
\left| \int_{r_0/\| z\|}^{1}\frac{\mathcal{R}f(tz)}{t}dt\right|
\\
&\leq
\frac{\| z\|}{r_0}\int_{r_0/\| z\|}^{1}\left|{\mathcal{R}f(tz)}\right|dt
\\
&\leq
\frac{\varepsilon \| z\|}{4Mr_0}\int_{r_0/\| z\|}^{1}\frac{1}{\omega(t\| z\|)}dt
\\
&\leq
\frac{\varepsilon }{2M}\int_{r_0}^{\| z\|}\frac{1}{\omega(t)}dt.
\end{align*}
Set $K=\sup_{\| z\|\leq r_0}|f(z)|$.
By Proposition \ref{p.growth}, $K<\infty$.
Then as in the proof of \cite[Theorem 3.2]{T07},
we have $T_{\varphi}f\in \mathcal{B}_{\mathcal{R}}(\B)_{\mu,0}$.

Conversely, assume that
$T_{\varphi}: \mathcal{B}_{\mathcal{R}}(\B)_{\omega,0}
\to \mathcal{B}_{\mathcal{R}}(\B)_{\mu,0}$
is bounded.
Since 
\[
\varphi(z)=\varphi(0)+\int_0^1\mathcal{R}\varphi(tz)\frac{1}{t}dt
=\varphi(0)+(T_{\varphi}1)(z),
\]
$\varphi \in \mathcal{B}_{\mathcal{R}}(\B)_{\mu,0}$.
Since the function $f_v$ defined in Lemma \ref{testfunction1}
belongs to  $\mathcal{B}_{\mathcal{R}}(\B)_{\omega,0}$,
we obtain \eqref{boundedcondition} by the proof of Theorem \ref{bounded}.
\end{proof}

From Theorems \ref{bounded} and \ref{bounded-little},
we obtain the following corollary which is a generalization of
the result on the Euclidean unit ball in $\C^n$
\cite[Corollary 3.1]{T07} to the unit ball of a complex Banach space.

\begin{corollary}
Let $\omega$ and $\mu$ be normal functions and let $\varphi \in H(\B)$.
Then $T_{\varphi}: \mathcal{B}_{\mathcal{R}}(\B)_{\omega,0}
\to \mathcal{B}_{\mathcal{R}}(\B)_{\mu,0}$
is bounded if and only if
$\varphi \in \mathcal{B}_{\mathcal{R}}(\B)_{\mu,0}$ and
$T_{\varphi}: \mathcal{B}_{\mathcal{R}}(\B)_{\omega}
\to \mathcal{B}_{\mathcal{R}}(\B)_{\mu}$
is bounded. 
\end{corollary}

\section{Compactness of extended Ces\`{a}ro operators}
\label{section-compact}
\setcounter{equation}{0}

In this section, we study the compactness of 
the extended Ces\`{a}ro operator 
$T_{\varphi}: \mathcal{B}_{\mathcal{R}}(\B)_{\omega}
\to \mathcal{B}_{\mathcal{R}}(\B)_{\mu}$
and 
$T_{\varphi}: \mathcal{B}_{\mathcal{R}}(\B)_{\omega,0}
\to \mathcal{B}_{\mathcal{R}}(\B)_{\mu,0}$.

The following lemma is
a generalization of
the result on the Euclidean unit ball in $\C^n$
\cite[Lemma 4.1]{T07} to the unit ball of a complex Banach space.
It can be proved
by a well-known argument
which uses Montel's theorem
(cf. \cite[Lemma 4.4]{BGLM17}).
We omit the proof.

\begin{lemma}
\label{lem-compact}
Let $\B$ be the unit ball of a complex Banach space.
Let $\omega$ and $\mu$ be normal functions and let $\varphi\in H(\B)$.
Then $T_{\varphi}: \mathcal{B}_{\mathcal{R}}(\B)_{\omega}
\to \mathcal{B}_{\mathcal{R}}(\B)_{\mu}$
is compact if and only if for any bounded sequence $\{ f_j\}$ in $\mathcal{B}_{\mathcal{R}}(\B)_{\omega}$
which converges to $0$ uniformly on any compact subset of $\B$,
we have
$\lim_{j\to\infty}\| T_{\varphi}f_j\|_{\mathcal{R}, \mu}=0$.
\end{lemma}

The following theorem is a generalization of
the result on the Euclidean unit ball in $\C^n$
\cite[Theorem 4.1]{T07} to the unit ball of 
a complex Banach space.
Blasco, Galindo, Lindstr\"om and Miralles \cite{BGLM17}
provided necessary and sufficient conditions for compactness of composition operators
on the space of Bloch functions on the unit ball of a complex Hilbert space
under additional relatively compactness assumptions on the set 
related to the composition symbol.
For $\varphi \in H(\B)$, we consider the set
\[
E_{\varepsilon, \rho}
=\{  z\in \B: \| z\|\leq\rho, \exists s\in [1,\rho^{-1}] \mbox{ s.t. }  \mu(sz)|\mathcal{R}\varphi(sz)|\geq\varepsilon\}
\]
and give the following compactness results of $T_{\varphi}$ under the assumption that
$E_{\varepsilon, \rho}$
is relatively compact in $\B$ for any $\varepsilon>0$ and $\rho\in (0,1)$.

\begin{theorem}
\label{compact1}
Let $\B$ be the unit ball of a complex Banach space.
Let $\omega$ and $\mu$ be normal functions and let $\varphi \in H(\B)$ be
such that the set
$E_{\varepsilon, \rho}$
is relatively compact in $\B$ for any $\varepsilon>0$ and $\rho\in (0,1)$.
Then
\begin{enumerate}
\item[$({\rm i})$]
Assume that $\int_0^1\frac{1}{\omega(t)}dt<\infty$.
Then  $T_{\varphi}: \mathcal{B}_{\mathcal{R}}(\B)_{\omega}
\to \mathcal{B}_{\mathcal{R}}(\B)_{\mu}$
is compact if and only if $\varphi \in \mathcal{B}_{\mathcal{R}}(\B)_{\mu}$.

\item[$({\rm ii})$]
Assume that $\int_0^1\frac{1}{\omega(t)}dt=\infty$.
Then $T_{\varphi}: \mathcal{B}_{\mathcal{R}}(\B)_{\omega}
\to \mathcal{B}_{\mathcal{R}}(\B)_{\mu}$
is compact if and only if 
\begin{equation}
\label{eq-compact1}
\lim_{\| z\|\to 1}\mu(z)|\mathcal{R}\varphi(z)|\int_0^{\| z\|}\frac{1}{\omega(t)}dt=0.
\end{equation}
\end{enumerate}
\end{theorem}

\begin{proof}
(i)
First, assume that $T_{\varphi}: \mathcal{B}_{\mathcal{R}}(\B)_{\omega}
\to \mathcal{B}_{\mathcal{R}}(\B)_{\mu}$
is compact.
Then it is bounded and therefore, $\varphi \in \mathcal{B}_{\mathcal{R}}(\B)_{\mu}$
by the proof of Theorem \ref{bounded}.

Conversely, assume that  $\varphi \in \mathcal{B}_{\mathcal{R}}(\B)_{\mu}$.
Since $\int_0^1\frac{1}{\omega(t)}dt<\infty$,
\eqref{eq-bounded} holds and therefore $T_{\varphi}: \mathcal{B}_{\mathcal{R}}(\B)_{\omega}
\to \mathcal{B}_{\mathcal{R}}(\B)_{\mu}$
is bounded by Theorem \ref{bounded}.
For any $\varepsilon>0$,
there exists $\rho\in (1/2,1)$ such that
\begin{equation}
\label{eq4.3}
\mu(z)|\mathcal{R}\varphi(z)|\int_{\rho}^{\| z\|}\frac{1}{\omega(t)}dt<\frac{\varepsilon}{3},
\quad
\rho<\| z\|<1
\end{equation}
holds.
Let $\{ f_j\}$ be a bounded sequence in $\mathcal{B}_{\mathcal{R}}(\B)_{\omega}$
which converges to $0$ uniformly on any compact subset of $\B$.
We may assume that $\| f_j\|_{{\mathcal{R}},\omega}\leq 1$.
Then $|f_j|\leq C_{\rho}$ for all $j$ and $\| z\|\leq \rho$ by 
Proposition \ref{p.growth}, where 
\[
C_{\rho}=C_4\left(
1+\int_0^{\rho}
\frac{1}{\omega(t)}dt
\right).
\]
There exists a positive integer $N$ such that 
\[
|f_j(w)|\leq \frac{\varepsilon}{3\| \varphi\|_{\mathcal{R},\mu}+1},
\quad
j>N, w\in E_{\varepsilon/(3C_{\rho}),\rho}.
\]
Therefore,  for $\| z\|\leq \rho$ and $t=1$ or for $\rho<\| z\|<1$ and $t=\rho/\| z\|$, we have
\begin{equation}
\label{bound3}
\mu(z)|\mathcal{R}\varphi(z)||f_j(tz)| 
<
\frac{\varepsilon}{3},
\quad j>N.
\end{equation}
For $j>N$ and $\rho <\| z\|<1$, by $\| f_j\|_{{\mathcal{R}},\omega}\leq 1$,
\eqref{eq4.3} and \eqref{bound3},
we have
\begin{align}
\lefteqn{\mu(z)|\mathcal{R}\varphi(z)||f_j(z)|}
\nonumber \\
\quad &  \leq
\mu(z)|\mathcal{R}\varphi(z)|
\left| f_j(z)-f_j\left(\rho\frac{z}{\| z\|}\right)\right|
+\mu(z)|\mathcal{R}\varphi(z)|
\left| f_j\left(\rho\frac{z}{\| z\|}\right)\right|
\nonumber \\
\quad &  \leq
\mu(z)|\mathcal{R}\varphi(z)|
\int_{\rho/\| z\|}^1|\mathcal{R}f_j(tz)|\frac{dt}{t}
+\frac{\varepsilon}{3}
\nonumber \\
\quad &  \leq
\mu(z)|\mathcal{R}\varphi(z)|
\frac{\| z\|}{\rho}\int_{\rho/\| z\|}^1\frac{1}{\omega(t\| z\|)}{dt}
+\frac{\varepsilon}{3}
\nonumber \\
\quad & \leq
2\mu(z)|\mathcal{R}\varphi(z)|
\int_{\rho}^{\| z\|}\frac{1}{\omega(t)}{dt}
+\frac{\varepsilon}{3}
\nonumber \\
\quad & <
\varepsilon.
\label{bound4}
\end{align}
From \eqref{bound3} and \eqref{bound4},
we obtain
$\| T_{\varphi}f_j\|_{\mathcal{R},\mu}<\varepsilon$ for $j>N$.
By Lemma \ref{lem-compact},
$T_{\varphi}: \mathcal{B}_{\mathcal{R}}(\B)_{\omega}
\to \mathcal{B}_{\mathcal{R}}(\B)_{\mu}$
is compact.

(ii) 
Assume that $T_{\varphi}: \mathcal{B}_{\mathcal{R}}(\B)_{\omega}
\to \mathcal{B}_{\mathcal{R}}(\B)_{\mu}$
is compact.
If $\varphi$ does not satisfy \eqref{eq-compact1},
then there exist $\varepsilon>0$ and a sequence $\{ z_j\}\subset \B$ such that
$\lim _{j\to \infty}\| z_j\|=1$ and
\begin{equation}
\label{eq-opposite}
\mu(z_j)|\mathcal{R}\varphi(z_j)|\int_0^{\| z_j\|}\frac{1}{\omega(t)}dt\geq \varepsilon,
\quad j=1,2,3,\dots.
\end{equation}
We may assume that $\| z_j\|>r_1$,
where $r_1$ is the constant in Lemma \ref{lem-g}.
Let $f_j(z)=F_{z_j}(z)$ for $z\in \B$,
where $F_{z_j}$ is the function defined in Lemma \ref{testfunction2}.
From Lemma \ref{testfunction2},
$\{ f_j\}$ is a bounded sequence in $\mathcal{B}_{\mathcal{R}}(\B)_{\omega,0}$
and $f_j\to 0$ uniformly on any compact subset of $\B$.
Then $\lim_{j\to \infty}\| T_{\varphi}f_j\|_{\mathcal{R}, \mu}=0$ by Lemma \ref{lem-compact}.
On the other hand,
by Lemmas \ref{lem-g}, \ref{Taylor} and \eqref{eq-opposite}, we have
\begin{align*}
\| T_{\varphi}f_j\|_{\mathcal{R}, \mu}
&=
\sup_{z\in\B}\mu(z)|\mathcal{R}\varphi(z)||f_j(z)|
\\
&\geq
\mu(z_j)|\mathcal{R}\varphi(z_j)||f_j(z_j)|
\\
&=
\mu(z_j)|\mathcal{R}\varphi(z_j)|
\int_0^{\| z_j\|^2}g(t)dt
\\
&\geq
\frac{C_1}{C_3}\mu(z_j)|\mathcal{R}\varphi(z_j)|
\int_0^{\| z_j\|}\frac{1}{\omega(t)}dt
\\
&\geq
\frac{C_1}{C_3}\varepsilon.
\end{align*}
This is a contradiction.
Thus, we obtain \eqref{eq-compact1}.

Conversely, assume that  \eqref{eq-compact1} holds. 
Then $\varphi\in \mathcal{B}_{\mathcal{R}}(\B)_{\mu,0}$
and
for any $\varepsilon>0$,
there exists $\rho\in (1/2,1)$ such that
\[
\mu(z)|\mathcal{R}\varphi(z)|\int_{0}^{\| z\|}\frac{1}{\omega(t)}dt<\frac{\varepsilon}{3},
\quad
\rho<\| z\|<1
\]
holds.
The rest of the proof is similar to the case (i).
This completes the proof.
\end{proof}

The following theorem is a generalization of
the result on the Euclidean unit ball in $\C^n$
\cite[Theorem 4.2]{T07} to the unit ball of 
a complex Banach space.

\begin{theorem}
\label{compact2}
Let $\B$ be the unit ball of a complex Banach space.
Let $\omega$ and $\mu$ be normal functions and 
let $\varphi \in H(\B)$
be
such that the set
$E_{\varepsilon, \rho}$
is relatively compact in $\B$ for any $\varepsilon>0$ and $\rho\in (0,1)$.
Then
$T_{\varphi}: \mathcal{B}_{\mathcal{R}}(\B)_{\omega,0}
\to \mathcal{B}_{\mathcal{R}}(\B)_{\mu,0}$
is compact if and only if 
\begin{equation}
\label{compactcondition}
\lim_{\| z\|\to 1}
\mu(z)|\mathcal{R}\varphi(z)|\int_0^{\| z\|}\frac{1}{\omega(t)}dt=0.
\end{equation}
\end{theorem}

\begin{proof}
Assume that \eqref{compactcondition} holds,
Then, 
by Theorems \ref{bounded-little} and \ref{compact1},
we obtain that 
$T_{\varphi}: \mathcal{B}_{\mathcal{R}}(\B)_{\omega,0}
\to \mathcal{B}_{\mathcal{R}}(\B)_{\mu,0}$
is compact.

Conversely, assume that
$T_{\varphi}: \mathcal{B}_{\mathcal{R}}(\B)_{\omega,0}
\to \mathcal{B}_{\mathcal{R}}(\B)_{\mu,0}$
is compact.
Then $\varphi \in \mathcal{B}_{\mathcal{R}}(\B)_{\mu,0}$
by Theorem \ref{bounded-little}.
Therefore, if $\int_0^1\frac{1}{\omega(t)}dt<\infty$,
then  \eqref{compactcondition} holds.
We consider the case $\int_0^1\frac{1}{\omega(t)}dt=\infty$.
If  \eqref{compactcondition} does not hold,
then there exist $\varepsilon>0$ and 
 a sequence $\{ z_j\}\subset \B$ such that
$\lim _{j\to \infty}\| z_j\|=1$ and
\begin{equation}
\label{eq-opposite2}
\mu(z_j)|\mathcal{R}\varphi(z_j)|\int_0^{\| z_j\|}\frac{1}{\omega(t)}dt\geq \varepsilon,
\quad j=1,2,3,\dots.
\end{equation}
We may assume that $\| z_j\|>r_1$,
where $r_1$ is the constant in Lemma \ref{lem-g}.
Let $f_j(z)=F_{z_j}(z)$ for $z\in \B$,
where $F_{z_j}$ is the function defined in Lemma \ref{testfunction2}.
From Lemma \ref{testfunction2},
$\{ f_j\}$ is a bounded sequence in $\mathcal{B}_{\mathcal{R}}(\B)_{\omega,0}$
and $f_j\to 0$ uniformly on any compact subset of $\B$.
Since $T_{\varphi}: \mathcal{B}_{\mathcal{R}}(\B)_{\omega,0}
\to \mathcal{B}_{\mathcal{R}}(\B)_{\mu,0}$
is compact,
we may assume that there exists some
$g\in \mathcal{B}_{\mathcal{R}}(\B)_{\mu,0}$ such that 
$\| T_{\varphi}f_j-g\|_{\mathcal{R},\mu}\to 0$
as $j\to \infty$.
Then for each $z\in \B$, we have
\[
g(z)=\lim_{j\to \infty}T_{\varphi}f_j(z)=T_{\varphi}(\lim_{j\to \infty}f_j)(z)=T_{\varphi}0(z)=0.
\]
Thus, we have $\| T_{\varphi}f_j\|_{\mathcal{R},\mu}\to 0$
as $j\to \infty$.
By the proof of Theorem \ref{compact1}, this contradicts with
\eqref{eq-opposite2}.
Thus,  \eqref{compactcondition} holds.
This completes the proof.
\end{proof}

From Theorems \ref{compact1} and \ref{compact2},
we obtain the following corollaries which are generalization of
the results on the Euclidean unit ball in $\C^n$
\cite[Corollaries 4.1 and 4.2]{T07} to the unit ball of 
a complex Banach space.

\begin{corollary}
Let $\B$ be the unit ball of a complex Banach space.
Let $\omega$ and $\mu$ be normal functions and let $\varphi \in H(\B)$
be
such that the set
$E_{\varepsilon, \rho}$
is relatively compact in $\B$ for any $\varepsilon>0$ and $\rho\in (0,1)$.
Assume that $\int_0^1\frac{1}{\omega(t)}dt=\infty$.
Then the following statements are equivalent:
\begin{enumerate}
\item[$({\rm i})$]
$T_{\varphi}: \mathcal{B}_{\mathcal{R}}(\B)_{\omega}
\to \mathcal{B}_{\mathcal{R}}(\B)_{\mu}$
is compact;
\item[$({\rm ii})$]
$T_{\varphi}: \mathcal{B}_{\mathcal{R}}(\B)_{\omega,0}
\to \mathcal{B}_{\mathcal{R}}(\B)_{\mu,0}$
is compact;
\item[$({\rm iii})$]
\[
\lim_{\| z\|\to 1}\mu(z)|\mathcal{R}\varphi(z)|\int_0^{\| z\|}\frac{1}{\omega(t)}dt=0.
\]
\end{enumerate}
\end{corollary}

\begin{corollary}
Let $\B$ be the unit ball of a complex Banach space.
Let $\omega$ and $\mu$ be normal functions and let $\varphi \in H(\B)$
be
such that the set
$E_{\varepsilon, \rho}$
is relatively compact in $\B$ for any $\varepsilon>0$ and $\rho\in (0,1)$.
Assume that $\int_0^1\frac{1}{\omega(t)}dt<\infty$.
Then $T_{\varphi}: \mathcal{B}_{\mathcal{R}}(\B)_{\omega,0}
\to \mathcal{B}_{\mathcal{R}}(\B)_{\mu,0}$
is compact if and only if
$\varphi \in \mathcal{B}_{\mathcal{R}}(\B)_{\mu,0}$.
\end{corollary}

\begin{remark}
Let $\B$ be the unit ball of a finite dimensional complex Banach space. 
Then $E_{\varepsilon,\rho}$ is relatively compact in $\B$
for any $\varepsilon>0$ and $\rho\in (0,1)$.
Thus, in the finite dimensional case, 
Theorems 4.2, 4.3 and Corollaries 4.4, 4.5 hold without the assumption 
that $E_{\varepsilon,\rho}$ is relatively compact in $\B$.
\end{remark}
\noindent
{\bf Acknowledgments.}
{Hidetaka Hamada was partially supported by JSPS KAKENHI Grant
Number JP16K05217.
}

\bibliographystyle{amsplain}

\end{document}